\documentclass[headsepline, auto-pst-pdf, a4paper ,openany, 11pt]{article}
\usepackage{geometry}                
\usepackage[parfill]{parskip}    
\usepackage{a4wide}
\usepackage[dvipdfm]{graphicx}
\usepackage{color,caption,subcaption}
\usepackage{amsxtra}
\usepackage{amsthm, amssymb,bbm,mathtools}

\DeclareGraphicsExtensions{.pdf,.png,.jpg}

\newtheorem{theorem}{Theorem}[section]
\newtheorem{lemma}[theorem]{Lemma}

\newtheorem{conjecture}[theorem]{Conjecture}

\newtheorem{problem}[theorem]{Problem}
\newtheorem{question}[theorem]{Question}
\theoremstyle{remark}

\newtheorem{definition}[theorem]{Definition}

\newcommand{\cE}{\ensuremath{\mathcal{E}}}

\newcommand{\cM}{\ensuremath{\mathcal{M}}}
\newcommand{\cH}{\ensuremath{\mathcal{H}}}

\newcommand{\eps}{\varepsilon}
\newcommand{\size}[1]{\left| #1 \right|}				

\title{What is Ramsey-equivalent to a clique?}
\author{
Jacob Fox\thanks{
    Department of Mathematics,
    Massachusetts Institute of Technology,
    Cambridge, MA 02139-4307.
    Email: {\tt fox@math.mit.edu}.
    Research supported by a Packard Fellowship, by a Simons Fellowship, by NSF grant DMS-1069197, by a Sloan Foundation Fellowship, and by an MIT NEC Corporation Award.}  \and 
Andrey Grinshpun\thanks{
    Department of Mathematics,
    Massachusetts Institute of Technology,
    Cambridge, MA 02139-4307.
    Email: {\tt agrinshp@math.mit.edu}.
    Research supported by a National Physical Science Consortium Fellowship.} \and
Anita Liebenau\thanks{
	Department of Computer Science, University of Warwick, Coventry CV4 7AL, UK. This research 
	was done when the author was affiliated with the Institute of Mathematics, Freie Universit\"at Berlin, 14195 Berlin, 		Germany. Email: {\tt a.liebenau@warwick.ac.uk}. The author was supported by the Berlin Mathematical School. 
	The author would like to thank the Department of Mathematics,
    	Massachusetts Institute of Technology, Cambridge, MA 02139-4307 for its hospitality where this work was 			partially carried out.} 
	\and
Yury Person\thanks{Institute of Mathematics, Goethe-Universit\"at, 60325 Frankfurt am Main, Germany. This research 
	was done when the author was affiliated with the Institute of Mathematics, Freie Universit\"at Berlin. Email: {\tt person@math.uni-frankfurt.de}} \and
Tibor Szab\'o\thanks{
Institute of Mathematics, Freie Universit\"at Berlin, 14195 Berlin, Germany. Email: {\tt szabo@math.fu-berlin.de}}
}
\begin{document}
\maketitle

\begin{abstract}
A graph $G$ is Ramsey for $H$ if every two-colouring of the edges of $G$ contains a monochromatic copy of $H$. Two graphs $H$ and $H'$ are Ramsey-equivalent if every graph $G$ is Ramsey for $H$ if and only if it is Ramsey for $H'$. In this paper, we study the problem of determining which graphs are Ramsey-equivalent to the complete graph $K_k$. A famous theorem of Ne\v set\v ril and R\"odl implies that any graph $H$ which is Ramsey-equivalent to $K_k$ must contain $K_k$. We prove that the only connected graph which is Ramsey-equivalent to $K_k$ is itself. This gives a negative answer to the question of Szab\'o, Zumstein, and Z\"urcher on whether $K_k$ is Ramsey-equivalent to $K_k\cdot K_2$, 
the graph on $k+1$ vertices consisting of $K_k$ with a pendent edge. 

In fact, we prove a stronger result.  A graph $G$ is Ramsey minimal for a graph $H$ if it is Ramsey for $H$ but no proper subgraph of $G$ is Ramsey for $H$. Let $s(H)$ be the smallest minimum degree over all Ramsey minimal graphs for $H$. 
The study of $s(H)$ was introduced by Burr, Erd\H{o}s, and Lov\'asz, where they show that $s(K_k)=(k-1)^2$. We prove 
that $s(K_k \cdot K_2)=k-1$, and hence $K_k$ and $K_k\cdot K_2$ are not Ramsey-equivalent.

We also address the question of which non-connected graphs are Ramsey-equivalent to $K_k$. Let $f(k,t)$ be the maximum $f$ such that the graph $H=K_k+fK_t$, consisting of $K_k$ and $f$ disjoint copies of $K_t$, is Ramsey-equivalent to $K_k$.  Szab\'o, Zumstein, and Z\"urcher gave a lower bound on $f(k,t)$. We prove an upper bound on $f(k,t)$ which is roughly within a factor $2$ of the lower bound.
\end{abstract}
\section{Introduction}

A graph $G$ is {\it $H$-Ramsey} or {\em Ramsey for $H$}, denoted by $G \rightarrow H$, if any two-colouring of the edges of $G$ contains a monochromatic copy of $H$. The fact that for every graph $H$ there is a graph $G$ such that $G$ is $H$-Ramsey was first proved by Ramsey \cite{R30} in 1930 and rediscovered independently by Erd\H{o}s and Szekeres a few years later \cite{ES35}. Ramsey theory is currently one of the most active areas of combinatorics with connections to number theory, geometry, analysis, logic, and computer science. 

A fundamental problem in graph Ramsey theory is to understand the graphs $G$ that are $K_k$-Ramsey, where $K_k$ denotes the complete graph on $k$ vertices. The Ramsey number $r(H)$ is the minimum number of vertices of a graph $G$ which is $H$-Ramsey. The most famous question in this area is that of estimating the Ramsey number $r(K_k)$. Classical results of Erd\H{o}s \cite{E47} and Erd\H{o}s and Szekeres \cite{ES35} show that $2^{k/2} \leq r(K_k) \leq 2^{2k}$. While there have been several improvements on these bounds (see, for example, \cite{C09}), despite much attention, the constant factors in the above exponents remain the same. Given these difficulties, the field has naturally stretched in different directions. Many foundational results were proved in the 1970s which showed the depth and breadth of graph Ramsey theory. For instance, a famous theorem of Ne\v set\v ril and R\"odl \cite{nesetril1976} states that for every graph $H$ there is a graph $G$ with the same clique number as $H$ such that $G \rightarrow H$. 

Szab\'o, Zumstein, and Z\"urcher \cite{szz2010} defined two graphs $H$ and $H'$ to be {\em Ramsey-equivalent} if for every graph $G$, $G$ is $H$-Ramsey if and only if $G$ is $H'$-Ramsey. The result of  Ne\v set\v ril and R\"odl \cite{nesetril1976} above  implies that any graph $H$ which is Ramsey-equivalent to the clique $K_k$ must contain a copy of $K_k$. In this paper, we study the problem of determining which graphs are Ramsey-equivalent to $K_k$. In other words, knowing that $G$ is Ramsey for $K_k$, what additional monochromatic subgraphs must occur in any two-colouring of the edges of $G$?

In \cite{szz2010} it was conjectured that, for large enough $k$, the clique $K_k$ is Ramsey-equivalent to $K_k \cdot K_2$, the graph on $k+1$ vertices consisting of $K_k$ with a pendent edge. We settle this conjecture in the negative, showing that, for all $k$, the graphs $K_k$ and $K_k \cdot K_2$ are not Ramsey-equivalent. Together with the above discussion, this implies the following theorem. 

\begin{theorem}\label{equivmustdis}
Any graph which is Ramsey-equivalent to the clique $K_k$ must be the disjoint union of $K_k$ and a graph of smaller clique number. 
\end{theorem}
It is therefore natural to study the following function. Let $f(k,t)$ be the maximum $f$ such that $K_k$ and $K_k+f\cdot K_t$ are Ramsey-equivalent, where $K_k + f \cdot K_t$ denotes the disjoint union of a $K_k$ and $f$ copies of $K_t$. It is easy to see \cite{szz2010} that $f(k,k)=0$ and $f(k,1) = R(K_k)- k$.  For $t\leq k-2$, Szab\'{o} et al.~\cite{szz2010} proved the lower bound
\begin{equation}\label{lowerbound} 
f(k,t) \geq \frac{R(k,k-t+1)-2k}{2t}, 
\end{equation}
where $R(k,s)$ is the {\em Ramsey number} denoting the minimum $n$ such that every red-blue edge-colouring of $K_n$ contains a monochromatic red $K_k$ or a monochromatic blue $K_s$.

We prove the following theorem which, together with \eqref{lowerbound}, determines $f(k,t)$ up to roughly a factor $2$. 
\begin{theorem} \label{THM:DisjointCliques}
For $k > t\geq 3$,
\[  f(k,t) \leq \frac{R(k,k-t+1)-1}{t}. \]
\end{theorem}

While our proof does not apply for $t=2$, we may get an upper bound on $f(k,2)$ by taking a complete graph on $R(k,k)$ vertices. This is Ramsey for $K_k$ by definition, but is not Ramsey for $K_k + f K_2$ for any $f$ larger than $\frac{R(k,k)-k}{2}$, since for such an $f$ the graph $K_k + f K_2$ has more than $R(k,k)$ vertices. This is within roughly a factor of $4$ of the lower bound.

A graph $G$ is {\em $H$-minimal} if $G$ is $H$-Ramsey but no proper subgraph of $G$ is $H$-Ramsey. We denote the class of all $H$-minimal graphs by $\cM(H)$. Note that $G$ is $H$-Ramsey if and only if $G$ contains an $H$-minimal graph, so determining the $H$-Ramsey graphs reduces to determining the $H$-minimal graphs.  Also, two graph $H$ and $H'$ are Ramsey-equivalent if and only if $\cM(H)=\cM(H')$. 

A fundamental problem of graph Ramsey theory is to understand properties of graphs in $\cM(H)$. For example, the minimum number of vertices of a graph in $\cM(H)$ is precisely the Ramsey number $r(H)$. 
Another parameter of interest is $s(H)$, the smallest minimum degree of an $H$-minimal graph. That is, 
\[ s(H) := \min_{G \in \cM(H)} \delta(G), \]
where $\delta(G)$ is the minimum degree of $G$.

It is a simple exercise to show \cite{fl2006} that for every graph $H$, we have 
\[ 2\delta(H)-1 \leq s(H) \leq r(H)-1.\]
Somewhat surprisingly, the upper bound is far from optimal, at least for cliques. Indeed, Burr, Erd\H{o}s, and Lov\'asz \cite{burr1976} proved that $s(K_k)=(k-1)^2$. This is quite notable, as the simple upper bound mentioned above is exponential in $k$. 

Szab\'o, Zumstein, and Z\"urcher \cite{szz2010}  proved that $s(K_k \cdot K_2) \geq k-1$, where $K_k \cdot K_2$ is the graph on $k+1$ vertices which contains a $K_k$ and a vertex of degree $1$. We prove the following theorem, showing that their lower bound is sharp. 

\begin{theorem}
\label{thm:cliqueEdge} For all $k \geq 2$,
\[ s(K_k \cdot K_2) = k-1.\]
\end{theorem}

Note that Theorem \ref{thm:cliqueEdge} implies that $K_k$ and $K_k \cdot K_2$ are not Ramsey-equivalent. Indeed, for $k=2$ this is trivial, and for $k \geq 3$ we have $(k-1)^2 = s(K_k) > s(K_k \cdot K_2)=k-1$. Hence, Theorem \ref{equivmustdis} is a corollary of Theorem \ref{thm:cliqueEdge}. 

\noindent {\bf Organization:}\, In the next section, we prove Theorem \ref{thm:cliqueEdge}, showing that $s(K_k \cdot K_2)=k-1$; this implies Theorem \ref{equivmustdis}. In Section \ref{sect:additionalcliques}, we prove Theorem \ref{THM:DisjointCliques} giving an upper bound on the maximum number $f=f(k,t)$ such that $K_k$ is Ramsey-equivalent to $K_k+f\cdot K_t$. The final section contains relevant open problems of interest. 

\noindent {\bf Conventions and Notation:}\, All colourings are red-blue edge-colourings, unless otherwise specified. For a graph $G$, we write $V(G)$ for the vertex set of $G$ and $v(G)$ for the number of vertices of $G$. 

\section{Hanging edges}\label{sect:hangingedges}
In this section, we study the minimum degrees of graphs that are $K_k \cdot K_2$-minimal. Our plan is to construct a graph $G$ that contains a vertex $v$ of degree $k-1$ which is ``crucial'' for $G$ to be $K_k \cdot K_2$-Ramsey. That is, $G \rightarrow K_k \cdot K_2$, but $G-v \nrightarrow K_k \cdot K_2$. Thus, any minimal $K_k \cdot K_2$-Ramsey subgraph $G' \subseteq G$ has to contain $v$ and hence have minimum degree at most $k-1$. We therefore obtain the upper bound for Theorem \ref{thm:cliqueEdge}. 

We now proceed to develop tools useful for proving Theorem \ref{thm:cliqueEdge}. The following theorem of Ne\v set\v ril and R\"odl \cite{nesetril1976} states that there is a $K_k$-free graph $F$ so that any two-colouring of the edges of $F$ has a monochromatic $K_{k-1}$.

\begin{theorem}\label{ramseyGadget}
For every $k \geq 2$ there is some graph $F$ so that $F$ is $K_k$-free and $F \rightarrow K_{k-1}$.
\end{theorem}

By a {\em circuit of length $s$} in a hypergraph $\cH =(V,\cE )$ we mean a sequence $e_1,v_1,e_2,v_2,\ldots,e_s,v_s$ of distinct edges 
$e_1,\ldots,e_s\in \cE$ and distinct vertices $v_1,\ldots,v_s\in V$ such that $v_j \in e_j \cap e_{j+1}$ for all $1\leq j <s$, and $v_s \in e_s \cap e_1$. In particular, if two distinct hyperedges intersect in two or more vertices, we consider this as a circuit of length 2. By the {\em girth} of a hypergraph $\cH$ we denote the length of the shortest circuit in $\cH$. The following lemma is proved in \cite{eh1966} by a now standard application of the probabilistic method \cite{AS00}.

\begin{lemma}\label{girthindep}
For all integers $k,m \geq 2$ and every $\epsilon > 0$ there is a $k$-uniform hypergraph of girth at least $m$ and independence number at most $\epsilon n$, where $n$ is the number of vertices in the hypergraph.
\end{lemma}

We will need a strengthening of Theorem \ref{ramseyGadget} which states that there is a $K_k$-free graph $F$ so that any two-colouring of the edges of $F$ has a monochromatic $K_{k-1}$ inside of every $\eps$ fraction of the vertices.

\begin{definition}
We write $F \overset{\eps}{\rightarrow} K_{k}$ to mean that for every $S \subseteq V(F)$, $|S| \geq \eps v(F)$ implies $F[S] \rightarrow K_{k}$.
\end{definition}

\begin{lemma} \label{strongRamseyGadget}
For every $\eps >0$ and $k \geq 2$ there exists a graph $F$ which is $K_{k}$-free and $F \overset{\eps}{\rightarrow} K_{k-1}$.
\end{lemma}
\begin{proof}
The case where $k=2$ is trivial, so we will assume that $k \geq 3$. Take $F_0$ to be as in Theorem \ref{ramseyGadget}. By Lemma \ref{girthindep} there is some $v(F_0)$-uniform hypergraph $\mathcal{H}=(V,\mathcal{E})$ of girth at least $4$ and independence number less than $\eps \size{V}$. We construct a graph $F$ on vertex set $V$. The edges of $F$ are created by placing a copy of $F_0$ inside of each hyperedge in $\mathcal{E}$.

Since $\mathcal{H}$ has girth at least $4$, any triangle of $F$ must be contained in a single hyperedge of $\mathcal{H}$. Therefore, the vertex set of any copy of $K_{k}$ in $F$ must be contained in a single hyperedge of $\mathcal{H}$ as well. However, a single hyperedge forms just a copy of $F_0$ in $F$ and $F_0$ has no copy of $K_k$, so $F$ has no copy of $K_{k}$.

Since $\mathcal{H}$ has independence number less than $\eps \size{V}$, any set $S$ of at least $\eps \size{V}$ vertices must contain some hyperedge. Hence, $F[S]$ contains a copy of $F_0$. As $F_0 \rightarrow K_{k-1}$, we also have $F[S] \rightarrow K_{k-1}$.
\end{proof}

From this $F$ we construct a gadget graph $G_0$ with a useful property, namely that a particular copy of $K_k$ is forced to be monochromatic.
\begin{lemma} \label{ourGadget1}
There exists a graph $G_0$ with a subgraph $H$ isomorphic to $K_k$ contained in $G_0$ such that 
\begin{enumerate} 
	\item there is a colouring of $G_0$ without a red $K_k \cdot K_2$ and without a blue $K_k$ and 
	\item every colouring of $G_0$ without a monochromatic copy of $K_k \cdot K_2$ 
		results in $H$ being monochromatic.
\end{enumerate}
\end{lemma}

In order to prove that $H$ must be monochromatic in the above lemma, we will employ a technique we call colour focusing.

\begin{lemma} (Focusing Lemma)
Let $G=(A\cup B, E)$ be a complete bipartite graph with a colouring 
$\chi : E \rightarrow \{ red, blue\}$ of its edges. Then there exist subsets $A'\subseteq A$ and  $B' \subseteq B$ with $|A'| \geq |A|/2$ and $|B'| \geq |B|/2^{|A|}$, such that  
\begin{itemize}
\item[(a)] for every vertex $a \in A$, the set of edges from $a$ to $B'$ is monochromatic, and
\item[(b)] $\chi$ is constant on the edges between $A'$ and $B'$.
\end{itemize}
\end{lemma}

\begin{proof}
Define, for some vertex $b \in B$, the colour pattern ${\bf c}_b$ of $b$ to be the function with domain $A$ that maps a vertex $a\in A$ to the colour of the edge $\{a,b\}$. Consider the most common of the $2^{|A|}$ possible colour patterns the vertices 
in $B$ might have towards $A$ and call it ${\bf c}$. We define $B' \subseteq B$ 
to be the set of vertices having colour pattern ${\bf c}$. By the pigeonhole principle 
$|B'| \geq |B|/2^{|A|}$. Now, each vertex $a$ in $A$ has only edges of colour ${\bf c}(a)$ to $B'$, which proves part (a).
Then, each vertex in $A$ has either only red or only
blue edges to $B'$. Therefore, for some colour $c \in \{ red, blue\}$, at least
half of the vertices in $A$ have only edges of colour $c$ towards $B'$. This is the set we 
choose to be $A'$, concluding the proof of part (b).
\end{proof}

We now use part (a) to prove Lemma \ref{ourGadget1}.

{\bf Proof of Lemma \ref{ourGadget1}}.
If $k=2$ then taking $G_0$ to be a single edge suffices. We will henceforth assume $k \geq 3$. Take $\eps = 2^{-k^2}$ and let $F_1,\ldots,F_{k-2}$ be copies of the graph $F$ from Lemma \ref{strongRamseyGadget}. Add complete bipartite graphs between any two of these copies. Add a copy $H$ of $K_k$ and connect it to every vertex in every $F_i$. The resulting graph is $G_0$ (see Figure \ref{fig:Picture0}). To show $G_0 \nrightarrow K_k \cdot K_2$, colour all edges inside every $F_i$ and inside $H$ red, and all the remaining edges blue. The largest red clique is $H$, with only blue edges leaving $H$. 
The $F_i$ are $K_k$-free, and any edge leaving $F_i$ is blue as well. 
Since the graph of blue edges is $(k-1)$-chromatic ($F_1,\dots,F_{k-2},H$ is a partition into independent sets), 
the largest blue clique has order $k-1$. This verifies (1).

\begin{figure} [tbp]
 \centering
\phantom{asdfasdfasdfaasdfasfasdfad}
 \includegraphics[width=0.36\textwidth]{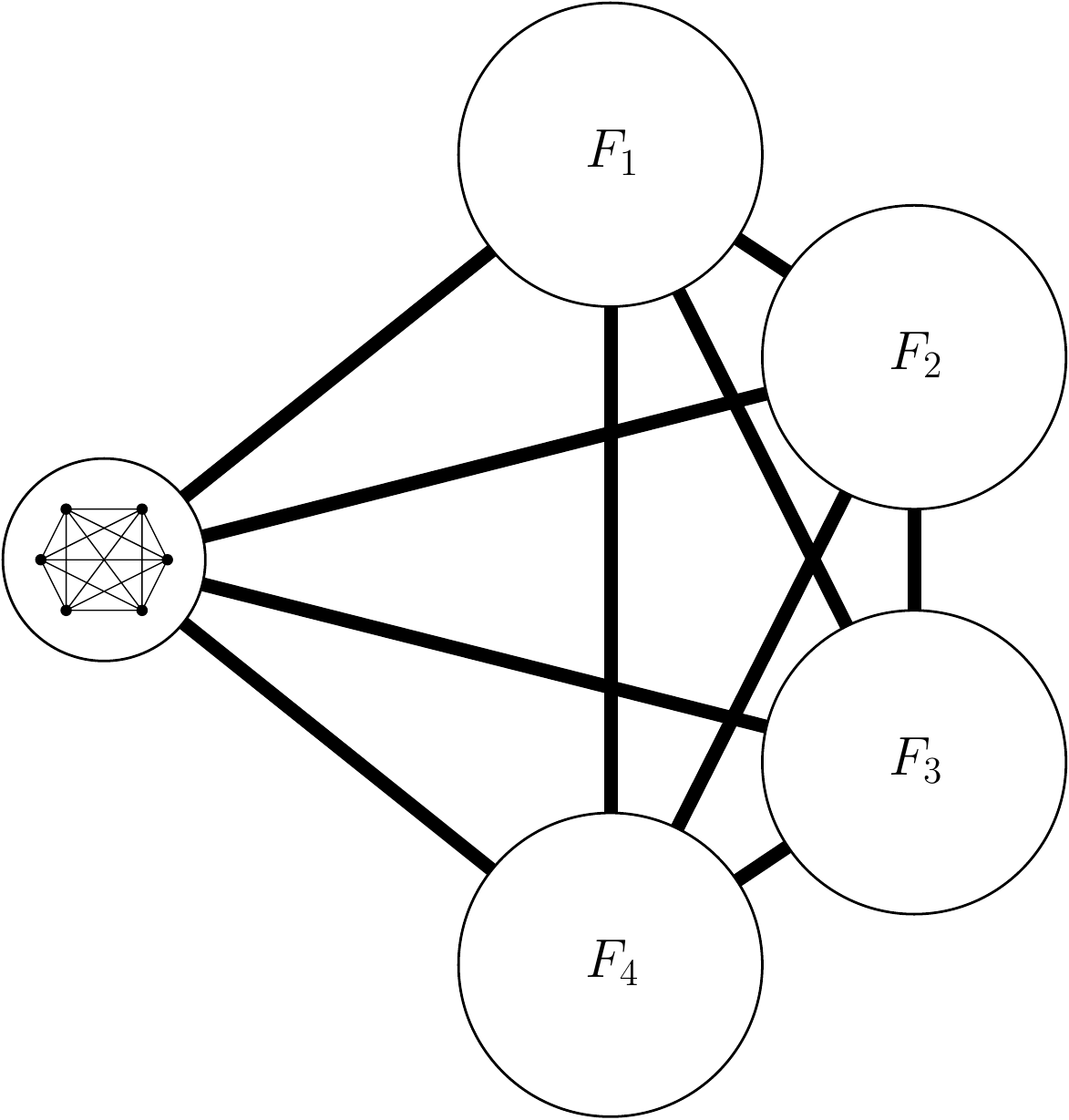}
 \caption{The gadget graph $G_0$ in Lemma \ref{ourGadget1} for $k = 6$. A thick line indicates that the vertices of the corresponding sets are pairwise connected.}
 \label{fig:Picture0}
 \end{figure}

 \begin{figure}[b]
 \centering
\phantom{asdfasdfasdfaasdfasfasdfad}
\begin{subfigure}[b]{0.45\textwidth}
	\centering
	 \includegraphics[width=0.6\textwidth]
	 {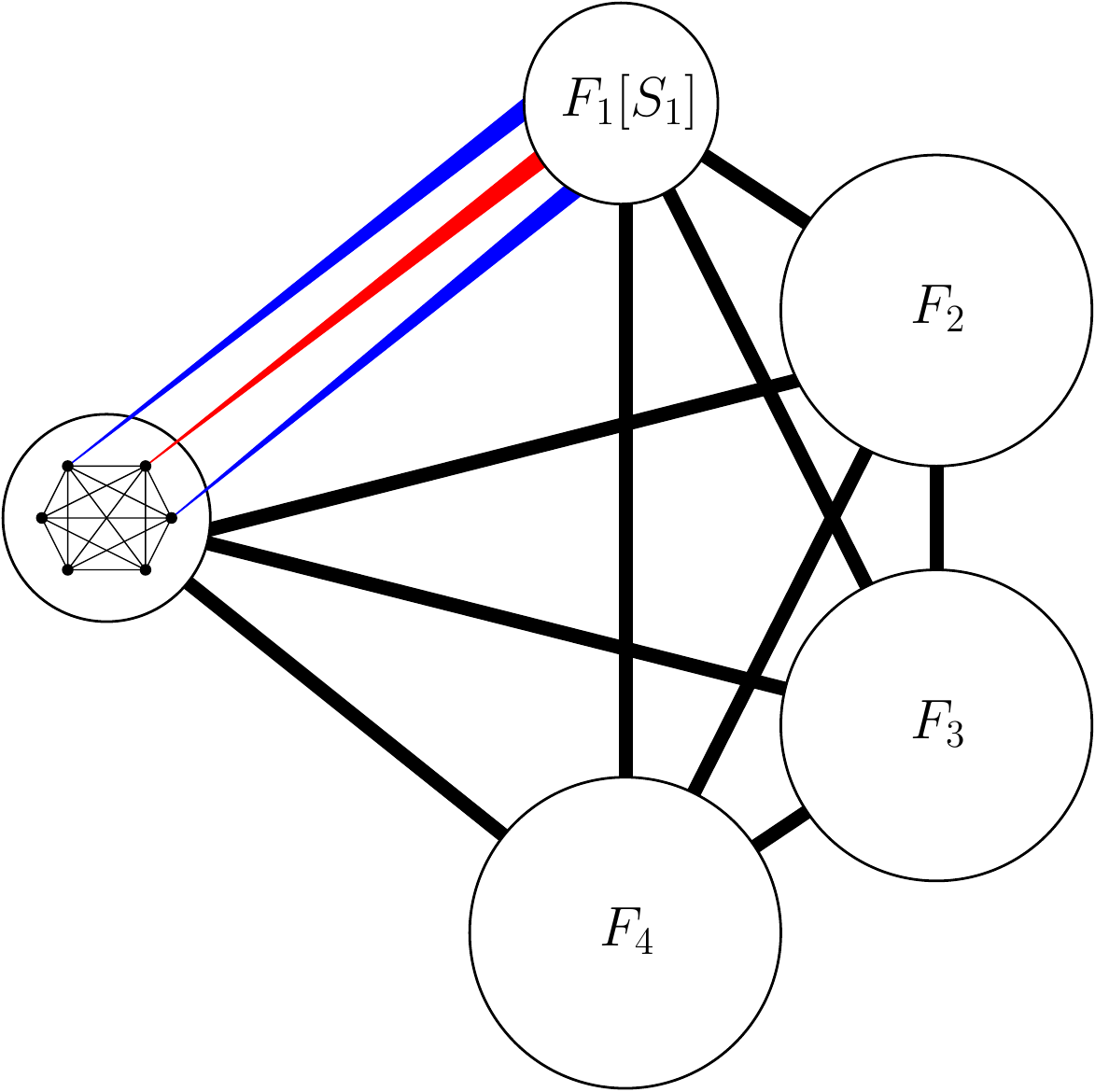}
	 \caption{Colour-focusing between $H=K_6$ and $F_1$.}
	 \label{fig:Picture1}
\end{subfigure} 
\qquad
 \begin{subfigure}[b]{0.45\textwidth}
 	\centering
	\includegraphics[width=0.6\textwidth]
	{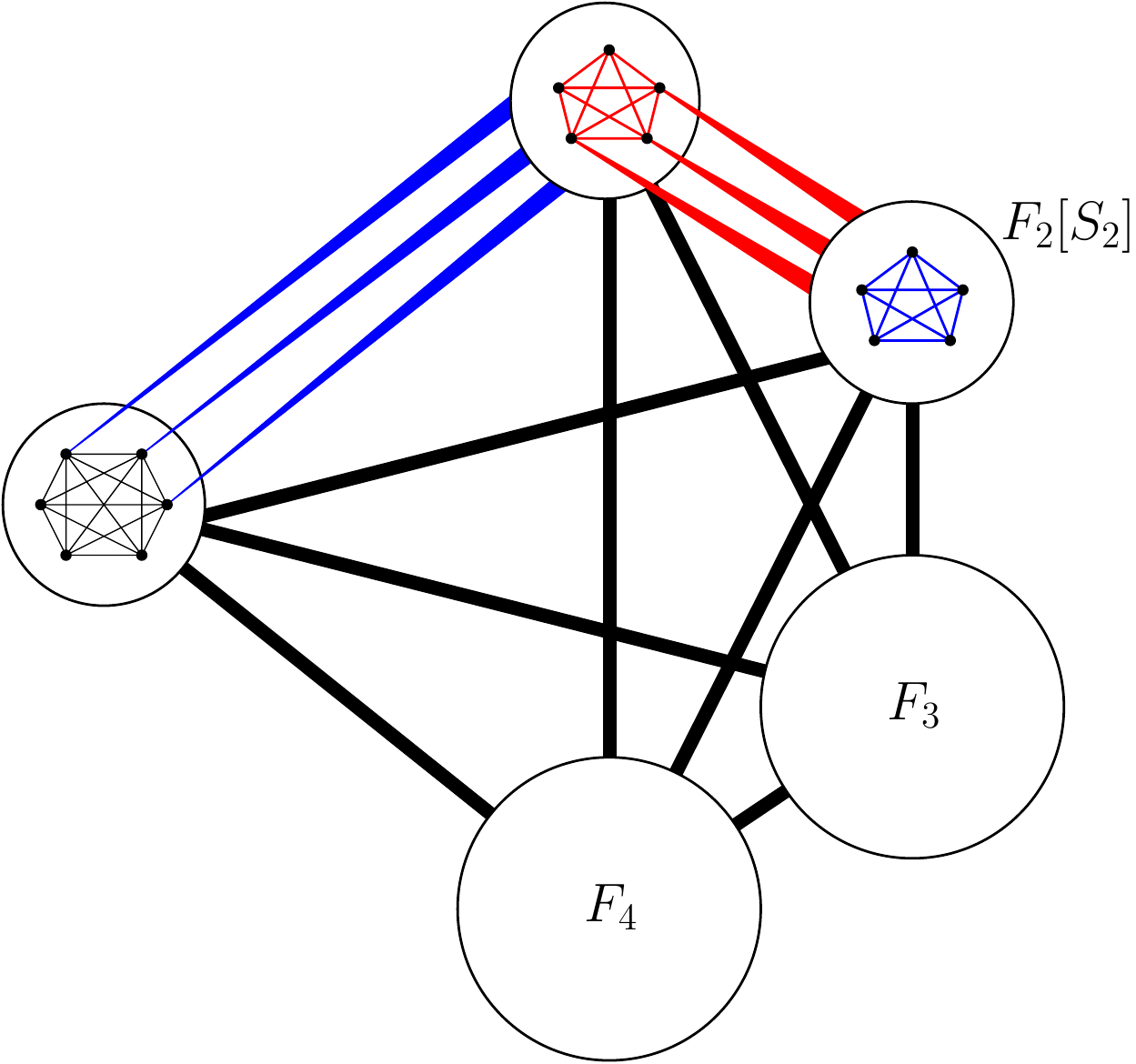}
	 \caption{There cannot be a blue $K_5$ in $F_2[S_2]$.}
	 \label{fig:Picture2}
 \end{subfigure}
 \caption{Illustrating the proof of Lemma \ref{ourGadget1}.}
 \label{fig:Pictures1and2}
 \end{figure}

For (2), assume $\chi$ is a red-blue colouring 
of the edges of $G_0$ without a monochromatic $K_k \cdot K_2$. We show that this forces $H$ to be monochromatic. By taking $A = V(H)$ and $B=V(F_1)$ in part (a) of the Focusing Lemma, we 
find a subset $S_1\subseteq V(F_1)$ such that $|S_1| \geq 2^{-k} v(F)$ and for each $a \in A$ the edges from $a$ to $S_1$ are monochromatic (see Figure \ref{fig:Picture1}). Then $|S_1| > \eps v(F)$, hence $F_1[S_1] \rightarrow K_{k-1}$.
Fix a monochromatic copy $H_1$ of $K_{k-1}$ contained in $S_1$, and assume without loss of generality that $H_1$ is red. 
We claim that all edges between $V(H)$ and $S_1$ (and in particular to $V(H_1)$) are blue. 
Indeed, if one vertex $i$ of $H$ had red edges to $S_1$, then $i$ along with $H_1$ and one (arbitrary) other vertex $v$ of $S_1$ would form a red copy of $K_k \cdot K_2$, a contradiction to our assumption on the colouring $\chi$.

We now iterate this argument. Assume we have found red cliques $H_1,\ldots,H_{t-1}$ in $F_1,\ldots,F_{t-1}$ with vertex sets $V_1,\ldots,V_{t-1}$, respectively, and that all the edges between these cliques as well as to $H$ are blue. By part (a) of the Focusing Lemma, in $F_t$ there is some subset $S_t \subseteq V(F_t)$ of the vertices of size at least $2^{-tk} v(F_t)$, so that each vertex $v \in V(H) \cup V_1 \cup V_2 \cup \cdots \cup V_{t-1}$ is monochromatic to $S_t$. Since $\size{S_t} > \eps v(F_t)$, we have $F_t[S_t] \rightarrow K_{k-1}$. We find a monochromatic copy of $K_{k-1}$ in $S_t$ and call it $H_t$. Assume for contradiction that $H_t$ is blue. In this case as before, all the edges between $H_t$ and $H$ as well as between $H_t$ and $H_1,\ldots,H_{t-1}$ would have to be red, otherwise there would be a blue $K_k\cdot K_2$. But if all these edges are red, then any two vertices of $H_t$ together with $H_1$ form a red $K_k\cdot K_2$ (see Figure \ref{fig:Picture2}). Hence, $H_t$ must be red, and as before all edges between $H_t$ and $H$ as well as between $H_t$ and $H_1,\ldots,H_{t-1}$ must be blue.

After applying this argument to $F_{k-2}$, we have a collection $H_1,\ldots,H_{k-2}$ of red $(k-1)$-cliques and 
complete bipartite blue graphs between any two of $H,H_1,\ldots,H_{k-2}$. 
Now, if some edge in $H$ were blue, this edge along with one vertex from each of $H_1,\ldots,H_{k-2}$ and any (arbitrary) other vertex from $H_1$ would create a blue $K_k \cdot K_2$.
Therefore, every edge of $H$ must be red, as desired.
\qed

The following lemma completes the proof of Theorem \ref{thm:cliqueEdge}.

\begin{lemma}\label{finalGadget100}
For every $k\geq 3$ there is a graph $G$ which contains a vertex $v$ of degree $k-1$ so that $G \rightarrow K_k \cdot K_2$ but $G-v \nrightarrow K_k \cdot K_2$.
\end{lemma}
\begin{proof}
Take $k-1$ copies $G_1,\ldots,G_{k-1}$ of the gadget graph $G_0$ from Lemma \ref{ourGadget1}, and let $H_1,\ldots,H_{k-1}$ be the copies of $K_k$ guaranteed to be monochromatic in any colouring without a monochromatic $K_k \cdot K_2$. Pick one vertex $v_i$ in each $H_i$, and insert all edges between the $v_i$, so they form 
a $K_{k-1}$.
In addition, pick an arbitrary vertex $v_k \neq v_2$ from $V(H_2)$ and insert an edge between it and $v_1$. Finally, add a vertex $v$ to the graph, and connect it to $v_1,\ldots,v_{k-1}$. This completes the construction of $G$ (see Figure \ref{fig:Picture3}). Clearly, $\deg(v) = k-1$.

\begin{figure} [htbp]
 \centering
\phantom{asdfasdfasdfaasdfasfasdfad}
 \includegraphics[width=0.6\textwidth]{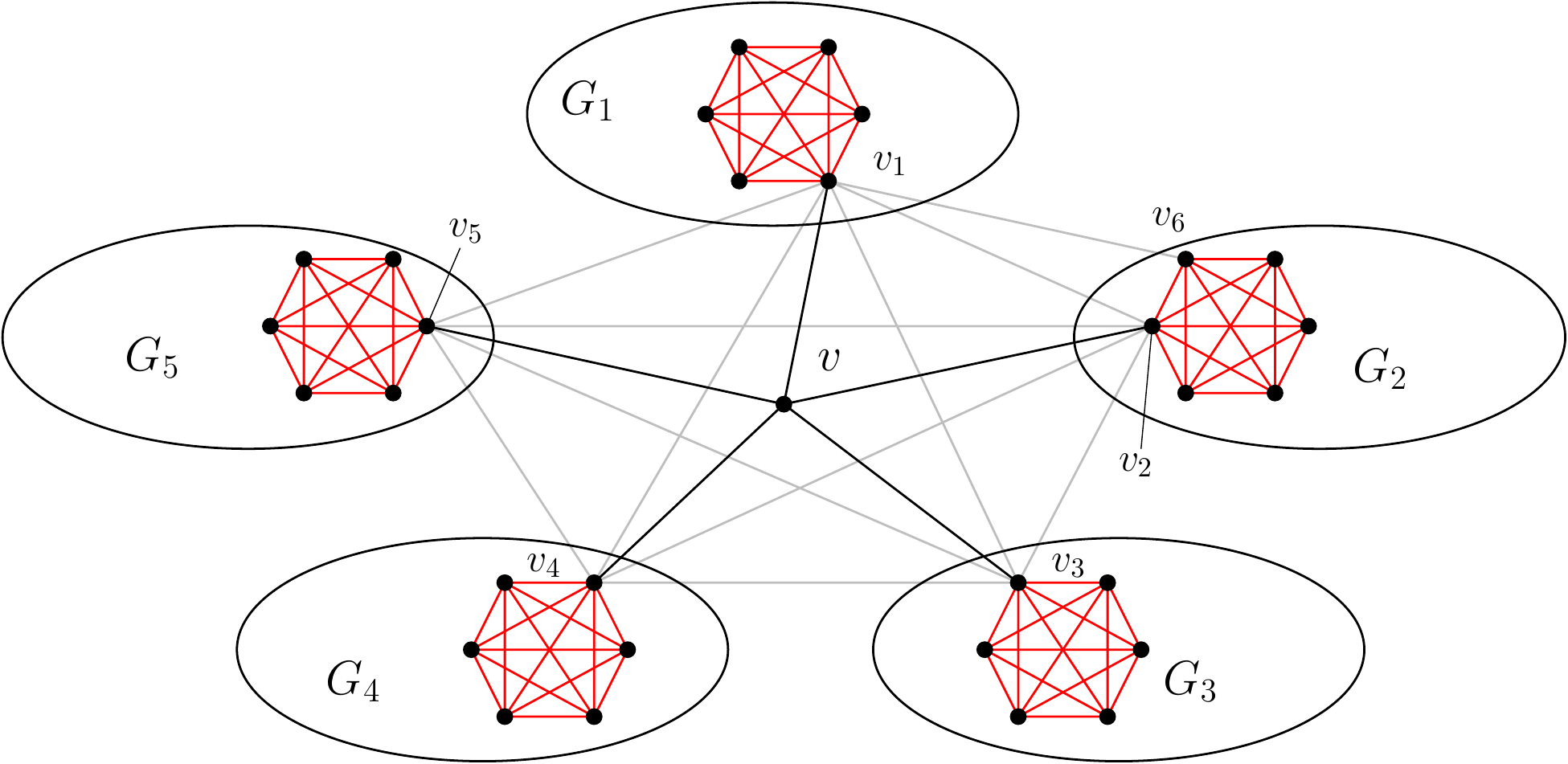}
 \caption{An example of the graph $G$ in Lemma \ref{finalGadget100} for $k = 6$.}
 \label{fig:Picture3}
 \end{figure}

To see that $G-v \nrightarrow K_k\cdot K_2$, colour each $G_i$ so it has no red $K_k\cdot K_2$ and no blue $K_k$. By property (2) of the gadget $G_0$ this also means that every $H_i$ is monochromatic red.
Colour the edges between $\{v_1,\ldots,v_{k-1}\}$ and the additional edge $\{v_1,v_k\}$ blue. Since none of the $G_i$ had a red $K_k\cdot K_2$ and we did not add any red edges, this colouring has no red $K_k \cdot K_2$. The $G_i$ have no blue $K_k$, and for $i =1, \ldots , k-1$ the vertex $v_i$ has no blue edges leaving $G_i$ except those to the other $v_j$. But the edge $\{ v_2, v_k\}$ is red, therefore there is no blue $K_k$ and in particular no blue $K_k \cdot K_2$.

Finally, we show that $G \rightarrow K_k \cdot K_2$. Let any colouring of $G$ be given, and suppose none of the copies of $G_0$ contains a monochromatic copy of $K_k\cdot K_2$. Then all of $H_1,\ldots,H_{k-1}$ are monochromatic. We claim they have the same colour. Indeed, if $H_i$ and $H_j$ had different colours, then the edge $v_iv_j$ would induce a monochromatic $K_k\cdot K_2$ with whichever copy of $K_k$ had the same colour as its own.

So all of the $H_i$ have the same colour; without loss of generality, let this colour be red. If any of the edges $v_iv_j$, for 
$1\leq i < j \leq k-1$ or for $i=1$, $j=k$ were red, 
then along with $H_i$ it would form a red $K_k \cdot K_2$. 
Similarly, if any of the edges $vv_i$ were red, then along with $H_i $ it would induce a red $K_k\cdot K_2$. Otherwise, all these edges are blue and then 
$v,v_1,\ldots,v_{k-1}$ and $v_k$ form a blue $K_k\cdot K_2$, as desired.
\end{proof}

\section{Clique with some disjoint smaller cliques}\label{sect:additionalcliques}
Recall that $K_k + f \cdot K_t$ denotes the disjoint union of 
a $K_k$ and $f$ copies of $K_t$. Also,  $f(k,t)$ is the largest number $f$ so that $K_k$ and $K_k + f \cdot K_t$ are Ramsey-equivalent. In this section, we prove Theorem \ref{THM:DisjointCliques}, which gives an upper bound on $f(k,t)$ for $t \geq 3$ and determines it up to roughly a factor of $2$. 

\begin{proof}[Proof of Theorem \ref{THM:DisjointCliques}]
Let $f=\left\lfloor \frac{R(k,k-t+1)-1}{t}\right\rfloor+1$. 
We will construct a graph $G$ 
with the following two properties. 
\begin{itemize}
\item[$(G1)$] $G \rightarrow K_k$ and 
\item[$(G2)$] $G \nrightarrow K_k + f \cdot K_t$.
\end{itemize}
{\bf Construction of $G$.\\}
$G$ will be constructed by combining a number of smaller graphs. For a positive integer $h$ and graphs $G_0,F_1,\ldots,F_{n_0}$ where $n_0$ is the number of vertices of $G_0$, define $G=G(h,G_0,F_1,\ldots,F_{n_0})$ as follows.

Take pairwise disjoint sets $V_H$ and $V_i$, $1\leq j\leq n_0$, such that $\size{V_H}=h$ and $|V_j| = |V(F_j)|$. Set $V:= V_H \cup \bigcup_{j=1}^{n_0} V_j$. Label the vertices of $H$ as $v_1,\ldots,v_h$.
The edge set $E$ is defined as follows. 
\begin{itemize}
\item $G[V_H] \cong K_h$, 
\item $G[V_j] \cong F_j$  for all $1\leq j \leq n_0$,
\item $v_iw \in E(G)$  for all $1\leq i \leq h$, and all $w\in \bigcup_{j=1}^{n_0} V_j$, 
\item  for all  $u\in V_i,\, w \in V_j$, $uw \in E(G)$ if and only if $ij\in E(G_0)$.
\end{itemize}
That is, our gadget graph consists of one copy of each $F_j$ together with a copy of a complete 
graph on $h$ vertices. Furthermore, we place a complete bipartite graph between $F_i$ and $F_j$ whenever 
$ij$ is an edge in $G_0$, and a complete bipartite graph between $V_H$ and $\bigcup_{j=1}^{n_0} F_j$ (see Figure \ref{fig:Picture4}).
\begin{figure} [tbp]
 \centering
\phantom{asdfasdfasdfaasdfasfasdfad}
 \includegraphics[width=0.6\textwidth]{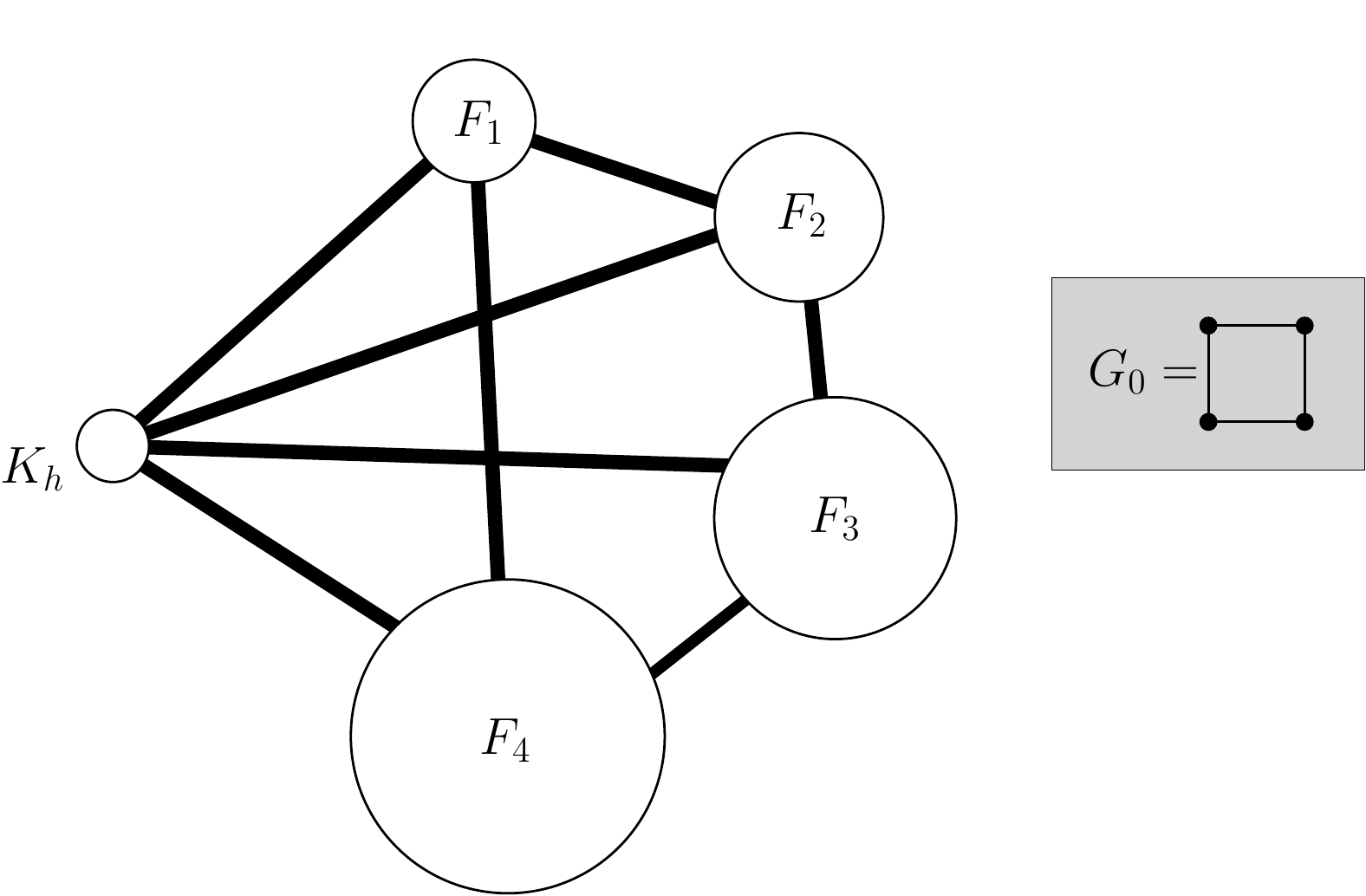}
 \caption{An illustration of $G(h,G_0,F_1,F_2,F_3,F_4)$ when $G_0 = C_4$. A thick line indicates that the vertices of the corresponding sets are pairwise connected.}
 \label{fig:Picture4}
\end{figure}

Set $h:= R(k,k-t+1)+k-1$ and $\eps_0:= 2^{-h-1}$. 
Let $G_0$ be a graph (given by Lemma \ref{strongRamseyGadget}) such that 
\begin{align}\label{propsG0}
&K_{k-1} \not\subseteq G_0 \quad\text{and}\quad G_0 \overset{\eps_0}{\longrightarrow} K_{k-2}.
\end{align} 
Now, set $n_0:= v(G_0)$ and assume without loss of generality that $V(G_0) = [n_0]$. 
For every $1\leq j \leq n_0$, we define $F_j$ iteratively. First, let 
$\eps_1 := 2^{-(h+n_0)}$
and let $F_1$ be a graph (given by Lemma \ref{strongRamseyGadget}) such that $K_t\not\subseteq F_1$ and
$F_1 \overset{\eps_1}{\longrightarrow} K_{t-1}$.
For $2\leq j \leq n_0$, assume we have defined $\eps_1,\ldots,\eps_{j-1}$ and $F_1,\ldots,F_{j-1}$. 
We then set 
\begin{align}\label{propsEpsJ}
\eps_j &:= 2^{-\left(h + n_0 -j + \sum_{i=1}^{j-1}v(F_i) \right)}
\end{align}
and let $F_j$ be a graph (given by Lemma \ref{strongRamseyGadget}) such that 
\begin{align}\label{propsFj}
& K_t\not\subseteq F_j \quad\text{and}\quad F_j \overset{\eps_j}{\longrightarrow} K_{t-1}.
\end{align}

Define the graph $G:=G(h,G_0,F_1,\ldots,F_{n_0})$, and take $V=V(G),E=E(G)$. Take $H$ to be the copy of $K_h$.

We now show that $G$ fulfills the two conditions $(G1)$ and $(G2)$ above.

{\bf The graph $G$ has property $(G2)$.\\}
To see that $G  \nrightarrow K_k + f \cdot K_t$, colour all edges inside $H$ and inside the copy of each $F_j$ red, 
and all edges between $H$ and $F_j$'s blue. 
Then the largest blue clique has size $k-1$ (since $G_0$ is $K_{k-1}$-free). 
So any monochromatic copy of $K_k + f \cdot K_t$ would need to be red. 
Since all the $F_j$'s are $K_t$-free, the red copy of $K_k + f \cdot K_t$ 
needs to lie inside $H$. 
However, $v(K_k + f \cdot K_t) = k+ft \geq k + R(k,k-t+1)> v(H)$. 
So $H$ cannot host a copy of $K_k + f \cdot K_t$.

{\bf The graph $G$ has property $(G1)$.\\}
Let $\chi:E\rightarrow \{\text{red, blue}\}$ be a 2-colouring of $G$. 
We apply a similar ``colour-focusing'' procedure as in the proof of Lemma \ref{ourGadget1}. 
This technique is used to obtain Lemma \ref{subgraphColourStructure}, which shows that there is a vertex subset for which the colouring is highly structured. From this lemma, it is not difficult to prove that there must be a monochromatic $K_k$.   

\begin{lemma}\label{subgraphColourStructure}
There exist a subset $J \subseteq [n_0]$ and subsets $W_j\subseteq V_j$ for each $j\in J$ 
such that the following holds. 
\begin{itemize}
\item[$(a)$] $|J| \geq n_0/2^h = 2\eps_0 n_0$, 
\item[$(b)$]  for all $j\in J$, $W_j$ is the vertex set of a monochromatic $K_{t-1}$ under $\chi$, 
\item[$(c)$]  for all $i,j\in J$ with $ij\in E(G_0)$, there exists $c_{ij}\in\{\text{red, blue}\}$ 
	such that for all $u\in W_i,\, w\in W_j$, $\chi (uw)=c_{ij}$.
\item[$(d)$]  for all $v_i\in V_H$, 
there exists $c_i \in \{\mbox{red, blue}\}$ such that for all $u\in \bigcup_{j\in J} W_j$, $\chi (v_iu)=c_i$.
\end{itemize}
\end{lemma}
\begin{figure} [bp]
 \centering
\phantom{asdfasdfasdfaasdfasfasdfad}
 \includegraphics[width=0.6\textwidth]{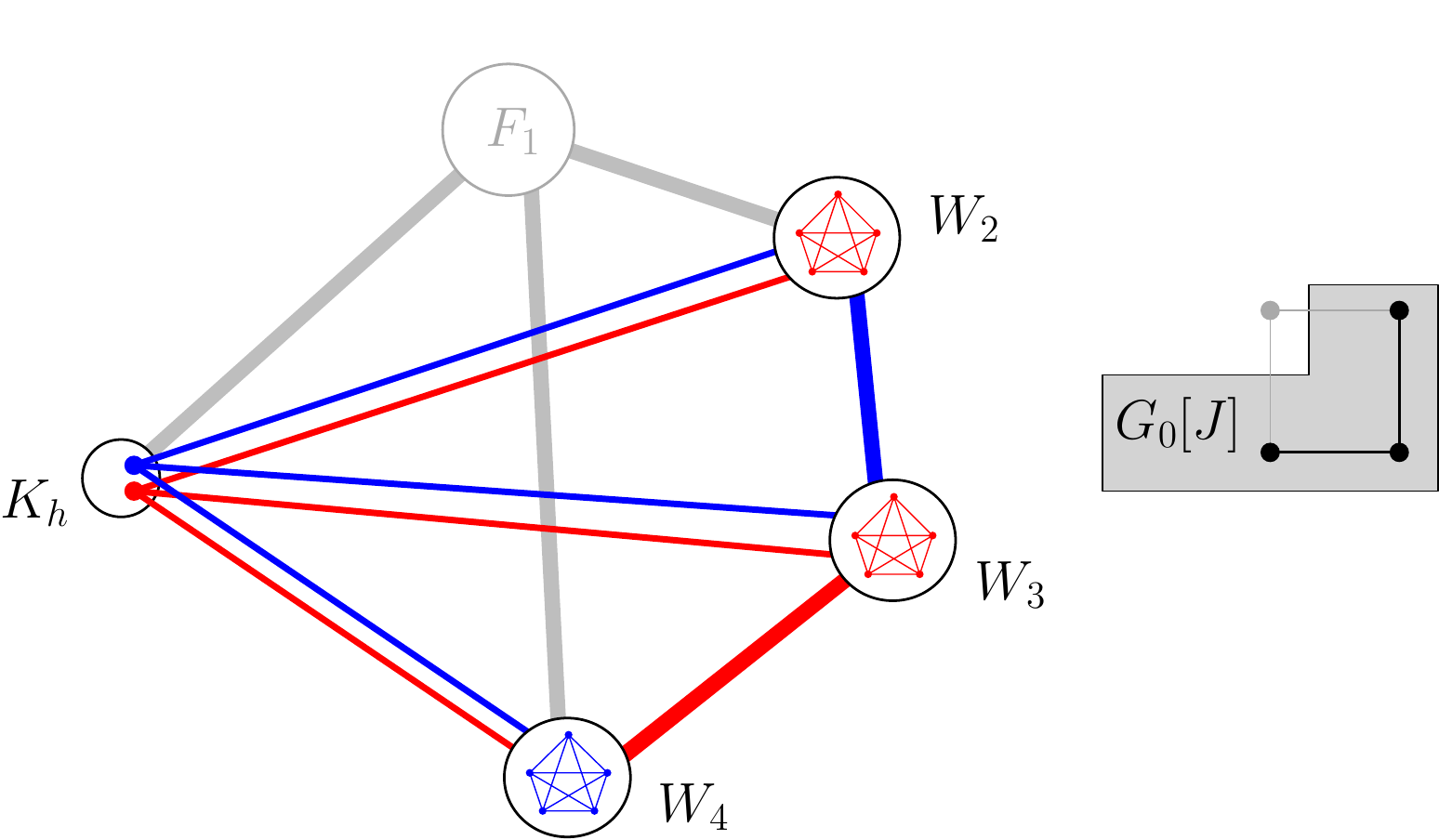}
 \caption{The colour patterns we find with Lemma \ref{subgraphColourStructure}.}
 \label{fig:Picture6}
 \end{figure}
The structure of the sets $J$ and $W_j$ in Lemma \ref{subgraphColourStructure} is depicted in 
Figure \ref{fig:Picture6}. Before proving the lemma, we first show how it implies that there is 
a monochromatic $K_k$ in $G$, which implies $(G1)$. 

{\bf Proof of $(G1)$ assuming Lemma \ref{subgraphColourStructure}.\\}
Let $J' \subseteq J$ with $|J'| \geq |J|/2$ be such that all $W_j$ with $j \in J'$ are monochromatic of the {\it same} colour. Consider the induced subgraph $G_0' := G_0[J']$ of $G_0$. Let $\chi'$ be the edge-colouring of 
$G_0'$ where each edge $ij\in E(G_0')$ has colour $\chi' (ij) := c_{ij}$. 
Since $|J'|\geq  |J|/2 \geq \eps_0 n_0$ by property $(a)$, 
and since $G_0  \overset{\eps_0}{\longrightarrow} K_{k-2}$ by definition of $G_0$, 
there exists a monochromatic copy of $K_{k-2}$ in $G_0'$ under $\chi'$. 
Let $I\subseteq J'$ denote the vertex set of this monochromatic copy, 
and assume without loss of generality that it is blue. 
Then, for all $i,j\in I$, $i\neq j$, the sets $W_i$ and $W_j$ are connected by 
complete bipartite graphs, all edges being blue under $\chi$. 
The monochromatic $W_j$ with $j \in I \subseteq J$ are all the same colour. If they were all blue, the union of the $W_j$ with $j \in I$,
each of which is of order $t-1$ by property $(b)$, form a monochromatic blue clique of order $(k-2)(t-1) \geq k$ 
(since $k > t \geq 3$), and thus there is a monochromatic $K_k$. Therefore, we may assume from now on that each $W_j$, $j\in I$, is a red $K_{t-1}$. 

Consider now the vertices in $V_H$. Any such vertex has either only red edges or only blue edges to $\bigcup_{j\in J'} W_j$, by property $(d)$. 
We call $v_i\in V_H$ {\em red} if $c_i = red$,
and {\em blue} otherwise. 
Suppose there exist two vertices, $v_i, v_j \in V_H$ which are both {\em blue}, 
such that $\chi (v_iv_j)=\text{blue}$. Then they form a blue $K_k$ with one vertex from each 
$W_j, j\in I$. So we can assume that for two {\em blue} vertices $v_i,v_j \in V_H$ we have 
$\chi (v_i v_j) = \text{red}$. 
But then we can also assume that there are at most $k-1$ {\em blue} vertices inside $H$, 
since otherwise they form a red $K_k$ inside $H$. 
So, there are at least $v(H)-(k-1) = R(k,k-t+1)$ {\em red} vertices $V_{red}\subseteq V_H$ in $H$. 
By definition of $R(k,k-t+1)$, $V_{red}$ contains either a red $K_{k-t+1}$ or a blue $K_k$. 
In the second case, we are done. In the first case, the vertex set $V_{red} \cup W_j$ 
contains a red $K_k$ for any $j\in I$, so we are done as well. \qed

{\bf Proof of Lemma \ref{subgraphColourStructure}.\\}
We prove the lemma in two steps. First, we apply part (a) of the Focusing Lemma with $A = V_H$  and each $V(F_j)$ as $B$ in order to ensure property $(d)$. 
Then, in order to ensure property $(c)$, we restrict to smaller and smaller sets inside 
$V(F_j)$ by repeatedly applying part (b) of the Focusing Lemma. 
These two steps are illustrated in Figure \ref{fig:picture7}. 

Recall that we are given a 2-colouring $\chi: E \rightarrow \{red, blue\}$ of the edge set of $G$.
First we show that there exists an index set $J \subseteq [n_0]$ and subsets $V_j'\subseteq V_j$ for each 
$j\in J$ such that the following properties hold. 
\begin{itemize}
\item[$(a)$] $|J|\geq n_0/2^h$,
\item[$(b')$] for all $j\in J$, $|V_j'|\geq v(F_j)/2^h$, and 
\item[$(d')$] for all $v_i\in V_H$, 
there exists $c_i \in \{red, blue\}$ such that for all $u\in \bigcup_{j\in J} V_j'$, $\chi (v_iu)=c_i$.
\end{itemize}

To see this, for each $j \in [n_0]$ apply part (a) of the Focusing Lemma 
to the complete  bipartite graphs between $V_H$ and $V_j$ 
to obtain subsets $V_j' \subseteq V_j$ of size at least $v(F_j)/2^h$ such that for each vertex $v \in V_H$ and $j\in [n_0]$, the set of edges between $v$ and $V_j'$ is monochromatic. 
In other words, for each index $j \in [n_0]$ there is a function ${\bf c}_j: V_h \rightarrow \{red,blue\}$ where ${\bf c}_j(v_i)$ is the colour of the edges from $v_i$ to $V_j'$. There are $2^h$ possible functions, so there must be a set $J \subseteq [n_0]$ of at least $n_0/2^{h}$ indices with a function ${\bf c}$ such that for any $j\in J$ we have ${\bf c}_{j} = {\bf c}$. Choosing $c_i:= {\bf c} (v_i)$ guarantees 
property $(d')$.

 \begin{figure}[hbt]
 \centering
\phantom{asdfasdfasdfaasdfasfasdfad}
\begin{subfigure}[b]{0.55\textwidth}
	\centering
	 \includegraphics[scale=0.55]
	 {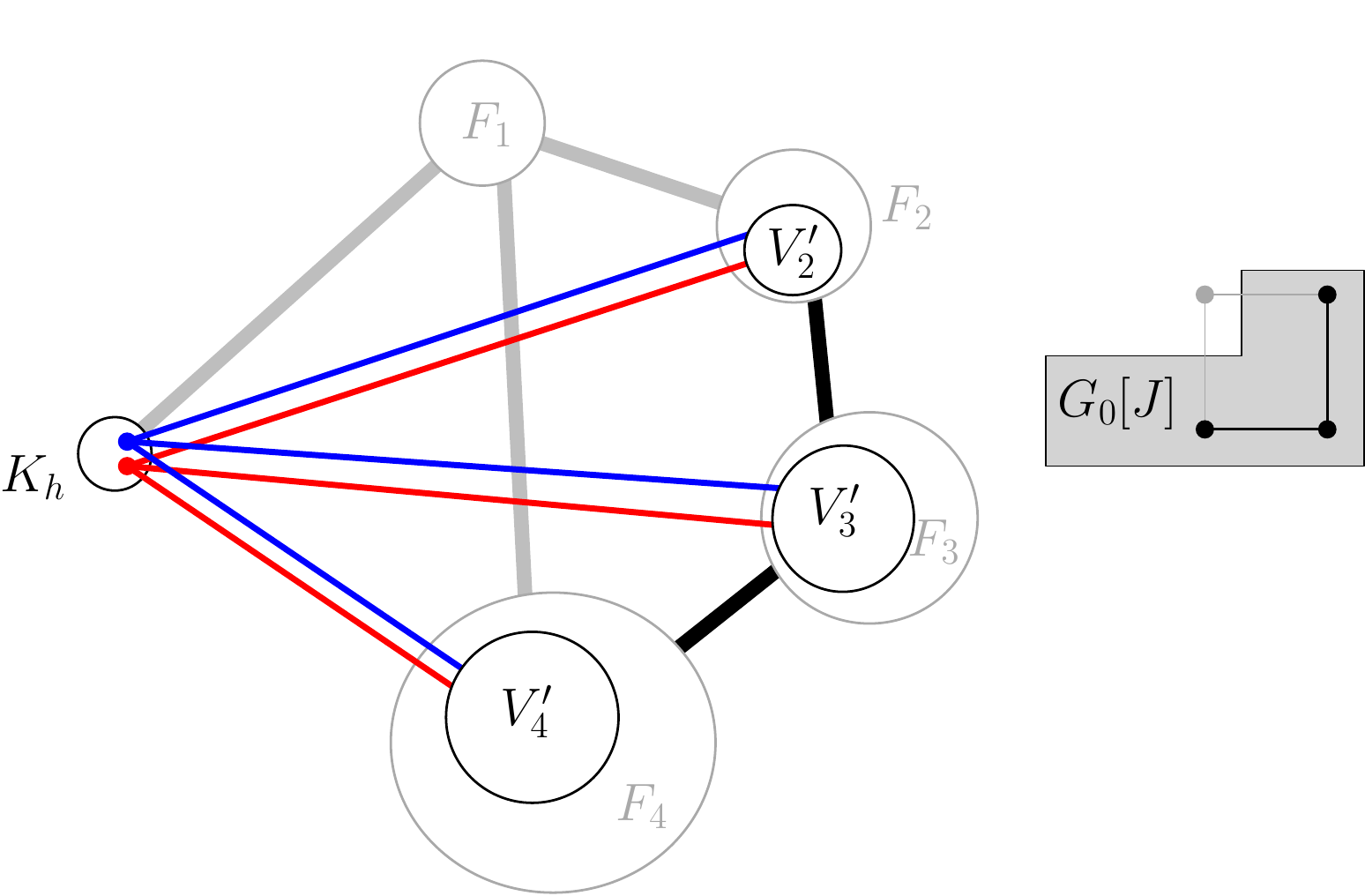}
	 \caption{Color-focusing of vertices of $K_h$.}
\end{subfigure} 
\qquad
 \begin{subfigure}[b]{0.35\textwidth}
 	\centering
	\includegraphics[scale=0.5]
	{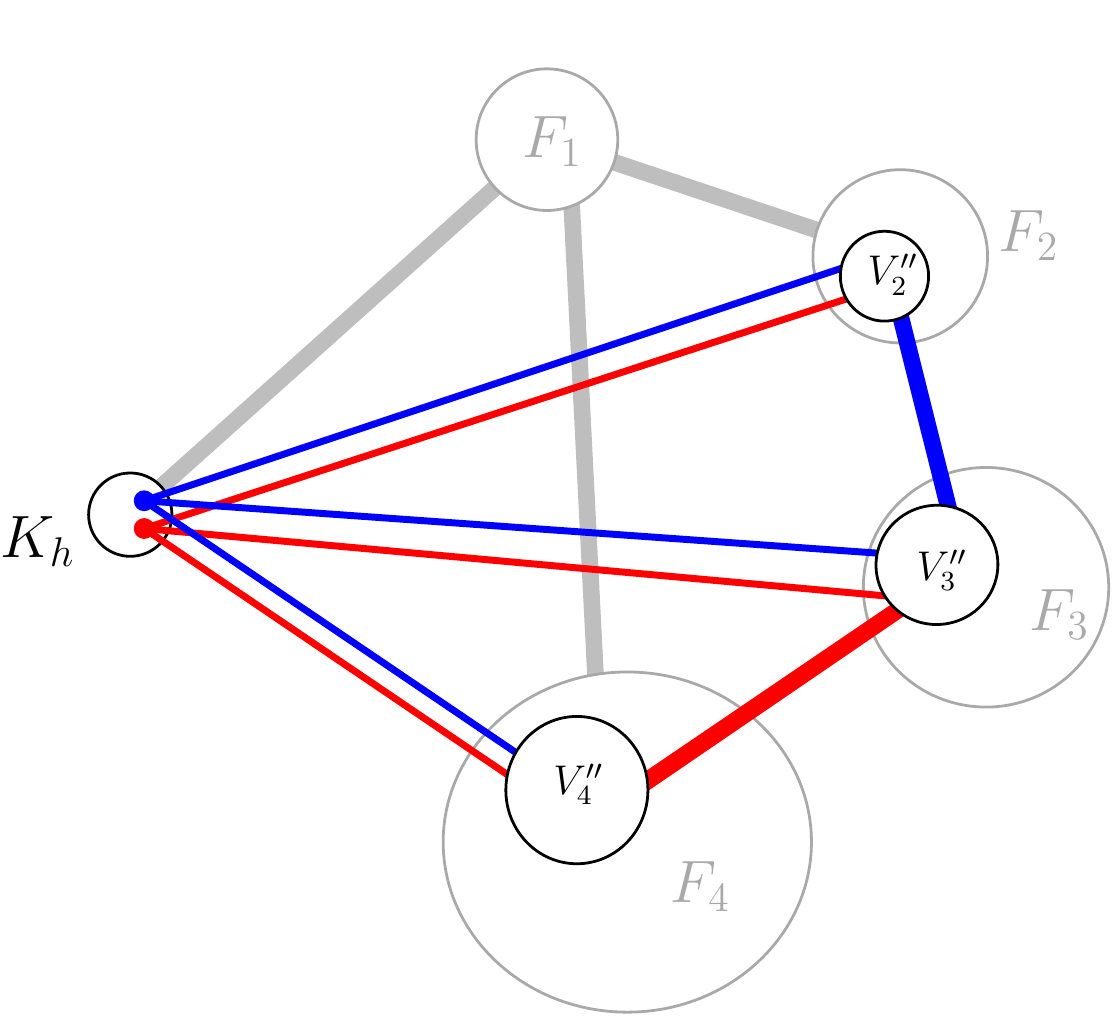}
	 \caption{Color-focusing of the complete bipartite graphs}
 \end{subfigure}
 \caption{The colour patterns we find in $G$.}
 \label{fig:picture7}
 \end{figure}
 
In the remainder of the proof we consider only the vertices in the sets $V_j'$ we have 
just defined. We will maintain subsets $V_j''\subseteq V_j'$, starting with $V_j''=V_j'$,  
and keep reducing their size  until the edges between them are monochromatically coloured for each pair. 

For ease of notation we assume $J=[\ell]$.
For each pair $i, j\in J$ with $ij\in E(G_0)$, we apply part (b) of the Focusing Lemma  
for the complete bipartite graph between the sets $V_i''$ and $V_j''$, where
$V_i''$ plays the role of $A$ and $V_j''$ plays the role of $B$ if $i < j$. 
These  applicatons are done one after another, in an arbitrary order, and 
after each of them the participating subsets $V_i''$ and $V_j''$ are redefined
to be the subsets $A'\subseteq A=V_i''$ and $B'\subseteq B=V_j''$ given by the Focusing Lemma. Hence,
after an application for the pair $i,j$, the edges between the sets 
$V_i''$ and $V_j''$ are monochromatic.  

Let $i\in J$ be an arbitrary index. The set $V_i''$ participates in an application of the 
Focusing Lemma $(i-1)$-times as the set $B$ and $(\ell -i)$-times as the set $A$.
The size of $V_i''$ might be reduced with 
each application, but the Focusing Lemma gives us a lower bound on the new size: 
it is at least half of the old size if $V_i''$ participated as $A$ and 
it is at least the $2^{-|V_j''|}$-fraction if $V_i''$ participated as $B$ together 
with some other set $V_j''$ as $A$ (with $j <i$).
Since we know how many times $V_i''$ participated as the set $A$ and how many 
times as the set $B$, we have a bound on its order at the end:
$$|V_i''|  \geq \frac{v(F_i)}{2^h} \cdot \frac{1}{2^{\ell -i}} \cdot \frac{1}{2^{\sum_{j=1}^{i-1} v(F_j)}}
\geq v(F_i)\cdot 2^{-\left(h + n_0 -i + \sum_{j<i}v(F_j) \right)} = \eps_i \cdot v(F_i),$$
where we used property $(b')$ to estimate the size of $V_i'$ at the beginning.

Since we applied the Focusing Lemma for every pair $i,j\in J,\ ij\in E(G_0)$, 
there exist $c_{ij} \in \{ red, blue\}$ such that the edges between $V_i''$ and $V_j''$ are monochromatic of colour $c_{ij}$, for every such pair.

It is now straight-forward to see that Lemma 
\ref{subgraphColourStructure} follows. 
Since each $F_i \overset{\eps_i}{\longrightarrow} K_{t-1}$ and by the above 
$V_i''$ at the end has size at least $\eps_i v(F_i)$, $V_i''$ does host a monochromatic $K_{t-1}$. 
Let $W_i$ be the vertex set of this $K_{t-1}$.
Now, since $W_i\subseteq V_i'' \subseteq V_i'$, $(c)$ and $(d)$ follow. 

As we saw earlier, the proof of Lemma \ref{subgraphColourStructure} completes the proof of Theorem \ref{THM:DisjointCliques}.
\end{proof}

\section{Open problems}\label{sect:openproblems}

Despite the progress made in this paper, we note the following interesting problems that remain open. 

Recall that $f(k,t)$ is the maximum $f$ such that $K_k$ and $K_k+f\cdot K_t$ are Ramsey-equivalent. We determined $f(k,t)$ up to roughly a factor $2$ for $k-1 > t > 2$. It would be of interest to close the gap between the lower and upper bounds. 

\begin{problem}
Determine $f(k,t)$. 
\end{problem}

A special case of this problem already asked in \cite{szz2010} is the following. Note that we have shown that 
$f(k,k-1)\leq 1$. That is, if $K_k$ and $K_k+K_{k-1}$ are Ramsey-equivalent, then $f(k,k-1)=1$ and otherwise $f(k,k-1)=0$. It is easy to see that $f(2,1)$ and $f(3,2)$ are $0$. We conjecture that for larger $k$ we have $f(k,k-1)=1$.

\begin{conjecture}
For $k$ at least $4$, $K_k$ and $K_k+K_{k-1}$ are Ramsey-equivalent.
\end{conjecture}

We proved that every graph (other than $K_k$) that is Ramsey-equivalent to $K_k$ is not connected. This naturally leads to the following question. 

\begin{question}
Is there a pair of non-isomorphic connected graphs $H_1,H_2$ that are Ramsey-equivalent?
\end{question}

An interesting special case of this question is about pairs of graphs such that one contains the other.
This motivates the following question.

\begin{question}
Is there a connected graph $H$ which is Ramsey-equivalent to a graph formed by adding a pendent edge to $H$?
\end{question}

We have recently shown \cite{FGLPS} that $K_{t,t}$ and $K_{t,t} \cdot K_2$, the graph formed by adding a pendent edge to $K_{t,t}$, are {\em not} Ramsey-equivalent. Furthermore, we proved $s(K_{t,t} \cdot K_2)=1$ while it was shown in \cite{fl2006} that $s(K_{t,t})=2t-1$.

We do not have a good understanding of how large of a connected subgraph can be added to $K_k$ and still be Ramsey-equivalent to $K_k$. For example, we have the following problem. 

\begin{problem}
Let $g(k)$ be the maximum $g$ such that $K_k$ is Ramsey-equivalent to $K_k+K_{1,g}$, the disjoint union of $K_k$ and the star $K_{1,g}$ with $g$ leaves. Determine $g(k)$.
\end{problem}

We only know that $g(k)$ is at least linear in $k$ and at most exponential in $k$. 
\bibliographystyle{abbrv}
\bibliography{referencesAll}

\begin{thebibliography}{10}

\bibitem{AS00}
N.~Alon and J.~Spencer.
\newblock {\em The probabilistic method, third edition}.
\newblock John Wiley {\&} Sons, 2000.

\bibitem{burr1976}
S.~Burr, P.~Erd{\H o}s, and L.~Lov{\'a}sz.
\newblock On graphs of {R}amsey type.
\newblock {\em Ars Combinatoria}, {\bf 1}:167--190, 1976.

\bibitem{C09}
D.~Conlon.
\newblock A new upper bound for diagonal {R}amsey numbers.
\newblock {\em Annals of Mathematics}, {\bf 170}:941--960, 2009.

\bibitem{E47}
P.~Erd{\H o}s.
\newblock Some remarks on the theory of graphs.
\newblock {\em Bulletin of the American Mathematical Society}, {\bf
  53}:292--294, 1947.

\bibitem{eh1966}
P.~Erd{\H{o}}s and A.~Hajnal.
\newblock On chromatic number of graphs and set-systems.
\newblock {\em Acta Mathematica Hungarica}, {\bf 17}:61--99, 1966.

\bibitem{ES35}
P.~Erd{\H o}s and G.~Szekeres.
\newblock A combinatorial problem in geometry.
\newblock {\em Compositio Mathematica}, {\bf 2}:463--470, 1935.

\bibitem{FGLPS}
J.~Fox, A.~Grinshpun, and A.~Liebenau.
\newblock Minimum degrees of minimal {R}amsey graphs for graphs with pendent
  edges.
\newblock {\em In Preparation}.

\bibitem{fl2006}
J.~Fox and K.~Lin.
\newblock The minimum degree of {R}amsey-minimal graphs.
\newblock {\em Journal of Graph Theory}, {\bf 54}:167--177, 2006.

\bibitem{nesetril1976}
J.~Ne{\v s}et{\v r}il and V.~R{\" o}dl.
\newblock The {R}amsey property for graphs with forbidden complete subgraphs.
\newblock {\em Journal of Combinatorial Theory, Series B}, {\bf 20}:243--249,
  1976.

\bibitem{R30}
F.~P. Ramsey.
\newblock On a problem of formal logic.
\newblock {\em Proceedings of the London Mathematical Society}, {\bf
  30}:264--286, 1930.

\bibitem{szz2010}
T.~Szab{\'o}, P.~Zumstein, and S.~Z{\"u}rcher.
\newblock On the minimum degree of minimal {R}amsey graphs.
\newblock {\em Journal of Graph Theory}, {\bf 64}:150--164, 2010.

\end{thebibliography}

\end{document}